\documentclass[12pt]{amsart}

\usepackage[T1]{fontenc}
\usepackage[sc,osf]{mathpazo}

\usepackage{mathrsfs,wasysym}
\usepackage{amsmath,amssymb,amsfonts,latexsym,graphicx}
\usepackage{eucal} 
\usepackage{fancyhdr,multicol}
\usepackage{pifont}
\usepackage{hyperref}

\usepackage{xcolor}
\usepackage{tikz}

\newlength{\originalbase}
\newcommand{\spacing}[1]{\setlength{\baselineskip}{#1\originalbase}}

\begin{document} 
\addtolength{\parskip}{6 pt}
\setlength{\parindent}{0 pt}
\spacing{1.1}

\newtheorem{theorem}{Theorem}
\newtheorem{proposition}[theorem]{Proposition}
\newtheorem{lemma}[theorem]{Lemma}
\newtheorem{corollary}[theorem]{Corollary}
\theoremstyle{definition}
\newtheorem{definition}[theorem]{Definition}
\newtheorem{example}[theorem]{Example}
\newtheorem{remark}[theorem]{Remark}
\newtheorem{exercise}[theorem]{Exercise}
\newcommand{\RR}{\mathbb{R}}
\newcommand{\QQ}{\mathbb{Q}}
\newcommand{\CC}{\mathbb{C}}
\newcommand{\ZZ}{\mathbb{Z}}
\newcommand{\NN}{\mathbb{N}}
\newcommand{\xx}{\mathbf{x}}
\newcommand{\vv}{\mathbf{v}}
\newcommand{\ww}{\mathbf{w}}
\newcommand{\kk}{\mathbf{k}}
\newcommand{\aaa}{\mathbf{a}}
\newcommand{\bbb}{\mathbf{b}}
\newcommand{\jjj}{\alpha}
\newcommand{\MM}{\mathcal{M}}
\newcommand{\lcm}{\mathrm{lcm}}
\newcommand{\var}{\mathrm{Var}}
\newcommand{\ord}{\mathrm{ord}}
\newcommand{\ndiv}{\nmid}
\newcommand{\nequiv}{\not\equiv}
\renewcommand{\qed}{\ding{113}}
\newcommand{\finial}{$$\mbox{\ding{167}}$$}

\input amssym.def
 
\input amssym.tex

\noindent
\thispagestyle{empty}
\begin{center}	  	 
{\large\bf What's for dessert?}\\[3mm]
{Tanya Khovanova and Daniel A.~Klain}
\end{center}
\vspace{6mm}

\section{Coffee House}

Dan and Tanya met in a coffeehouse and decided to have dessert. Both were watching their calories, so they decided to share.  
They would like to find a dessert that they will both enjoy, and to do so quickly, with a minimum of negotiation or calculation.
How should they choose which one? 

Tanya prefers pannacotta, and Dan prefers vanilla ice cream.  They can flip a coin, but Tanya hates vanilla ice cream and won't touch it. So this algorithm is not very good.

They can rank all the available desserts and pick the one with the smallest sum of ranks; call this the \textit{brute-force algorithm}.\footnote{This brute-force method is also known as {\em Borda Count} \cite{borda}\cite[p.~43]{saari}\cite[p.~65]{Szpiro}.}  We call the resulting dessert \textit{the optimal choice}. For example, if the chocolate mousse is Tanya's second choice and Dan's third choice, then the sum of its ranks is $5$. If the pannacotta is Tanya's first choice and Dan's second choice, then its sum of ranks is $3$. If you remember, Dan's first choice is vanilla ice cream, which is not even on Tanya's radar, making the pannacotta the optimal choice.

Let's have a formal discussion. Suppose the number of desserts is $N$. We will also assume that Tanya's and Dan's preferences are random and independent. What is the expected sum of ranks for the optimal choice? A calculation leads to the sequence A129591 in the OEIS \cite{OEIS}. For each permutation $p$ of numbers $1$ through $N$, define $M(p)$ as the minimum $p(i) + i$ over all indices $i$. Then A129591($N$) is the sum of $M(p)$ of all $p$. The sequence starts as $2$, $5$, $17$, $75$, $407$, $2619$, $19487$, $164571$, $1555007$, $16252779$. The closed formula for this sequence is 
\[\sum_{i=0}^{N-1} (N-i+1)i!((i+1)^{N-i}-i^{N-i}).\]

To get the expected minimum sum of ranks, we need to divide A129591($N$) by $N!$, leading to the sequence:
\begin{align}\label{borda}
2,\;\; 2.5,\;\; 2.8333,\;\;3.125,\;\; 3.3917,\;\; 3.6375,\;\; 3.8665,\;\; 4.0816, \;\; 4.2852,\;\; 4.4788, \;\; \ldots
\end{align}
and so on.
For example, if there are two desserts: $N=2$, then with probability $\frac{1}{2}$, we have the same preferences, and our minimum sum of ranks is $2$. Otherwise, our minimum sum of ranks is $3$. So the expected value is $2.5$.

However, the brute-force algorithm has many practical problems. For starters, it takes a long time and doesn't suit the ambiance of the coffee shop. It is not only about time. It is annoying to rank all the desserts.  While most people can easily identify their top and bottom choices, putting the middle choices in order feels like a waste of time.

The brute-force algorithm also can produce a tie, although this is actually a minor issue.

An example illustrates a more serious issue. 

\begin{example}
\label{example:badbf}
Without loss of generality, assume that the ranks for Dan are 1, 2, 3, 4, 5, and 6. Suppose the ranks for Tanya are 5, 6, 4, 3, 2, and 1, correspondingly.   
\end{example}

The brute-force algorithm will select (1,5), since this has the lowest sum. The result is very unfair to Tanya: it is her second to last choice. Although this algorithm guarantees that the total rank is not more than $N+1$, since this is the average total rank, it could still be very unattractive for Dan or Tanya.

If fairness is a priority, one might suggest another algorithm: to pick the dessert with the same ranks for both people or as close to the same rank as possible.

\begin{example}
\label{example:badfair}
We again assume that the ranks for Dan are 1, 2, 3, 4, 5, and 6. Suppose the ranks for Tanya are 5, 1, 2, 3, 4, and 6.   
\end{example}

Minimizing the difference for fairness is really bad here, yielding (6,6), which is miserable for both Dan and Tanya.

\section{The Better Algorithm}

In order to make a choice that is fair to both parties, Tanya suggests a much simpler algorithm: Dan reveals his top $K$ desserts from the menu. Tanya will then choose one of those $K$ desserts to share.

Consider Example~\ref{example:badbf}. Suppose $K=3$, so that this method results in ranks (3,4).  This is not so bad for either Dan or Tanya.

Now we formally explain what ``not so bad'' means. In the worst case for Dan, he would get his $K$th choice. We will soon see that the optimal $K$ is on the order of the square root of $N$, so Dan should be satisfied with his dessert. The worst rank for Tanya is $(N-K+1)$, which is worse than Dan's for optimal $K$. However, the probability of this happening is really low: $\frac{1}{\binom{N}{K}}$. It could happen that both Dan and Tanya are unlucky and get their worst choices possible with this algorithm. 
But this happens with a small probability:
\[\frac{1}{K\binom{N}{K}}.\]
Even in this worst case scenario the total rank is $N+1$, which is no worse than the average random choice. 

We expect this algorithm to work much better in most instances. But first, let's see how we should choose the value $K$.

\section{How should we choose $K$?}

Which value of $K$ provides the best result?

Consider the extreme cases. Suppose $K = 1$. Then Dan gets his first choice, and the expected minimum sum of ranks is $1 + \frac{N+1}{2}$. This matches the optimal choice for $1$ or $2$ desserts. Otherwise, it is worse than the optimal choice. On the other hand, suppose $K = N$. Then Tanya gets her first choice, and the expected sum of ranks is $1 + \frac{N+1}{2}$ again.  The best $K$ is likely to be somewhere in the middle. 

Suppose $N=3$ and $K = 2$. Then Dan gets his first choice with probability $\frac{1}{2}$ and the second with the same probability. So his expected rank is $\frac{3}{2}$. Tanya gets her first choice with a probability of $\frac{2}{3}$ and her second choice with a probability of $\frac{1}{3}$. So her expected rank is $\frac{4}{3}$. So our expected minimum sum of ranks is $\frac{17}{6}$. Surprisingly, it is the same as the best (see the sequence~(\ref{borda}) above). We are not this lucky for larger $N$, as we will soon see.

So what value of $K$ is closest to the optimal?

We can assume that Dan and Tanya are chosen randomly from a population that is uniformly distributed with all possible tastes. The problem now simplifies to the following:  Dan chooses a random subset of size $K$ from the set $\{1,2,\ldots, n\}$ (the Tanya rankings), after which Tanya chooses the minimum of that set.

Let $T$ denote the rank of the chosen dessert according to Tanya, and let $D$ denote the rank of that dessert according to Dan.

The sample space of Dan's choices consists of all possible subsets of size $K$.  The event that Tanya's final choice  has rank $d$ requires we choose $K-1$ more elements from the higher values $\{d+1,d+2, \ldots, N\}$. The probability that Tanya would rank the chosen dessert with value $d$ is then given by
\begin{align}
P_T(d) = \frac{\binom{N-d}{K-1}}{\binom{N}{K}}.
\label{tanya}
\end{align}

The following combinatorial identity is helpful for the analysis of the random variable $T$  (see also \cite[p.~121]{vanlint}).
\begin{proposition} For integers $N \geq 0$ and $0 \leq a, b \leq N$, 
\begin{align*}
\sum_{d=0}^N \binom{d}{a}\binom{N-d}{b} \;=\; \binom{N+1}{a+b+1}.
\end{align*}
\label{comb-id}
\end{proposition}

\begin{proof} There are $\binom{N+1}{a+b+1}$
ways to choose a subset of size $a+b+1$ from the set $\{1,2,\ldots, N+1\}$. 
Meanwhile, there are $\binom{d}{a}\binom{N-d}{b}$ ways to choose this subset 
so that its $(a+1)$st element is $d+1$.  \sloppy
\end{proof}

To compute the expected value of Tanya's rank, observe that
\begin{align} \notag
E(T) &\;=\; \sum_{d=1}^{N-K+1} dP_T(d) 
\;=\; \sum_{d=1}^{N-K+1} d\binom{N-d}{K-1}\binom{N}{K}^{-1} \\[1mm] \notag
&\;=\; \binom{N}{K}^{-1} \sum_{d=0}^{N} \binom{d}{1}\binom{N-d}{K-1} \quad \hbox{(the extra terms are all zero)}\\[1mm] \notag
&\;=\; \binom{N}{K}^{-1} \binom{N+1}{K+1}   \quad \hbox{(set $a=1$ and $b=K-1$ in Proposition~\ref{comb-id})}\\[1mm]
&\;=\; \frac{N+1}{K+1}.
\label{etanya}
\end{align}

Meanwhile, since Dan's tastes are uniformly random relative to Tanya's (and vice versa),
Tanya is equally likely to favor any of Dan's initial $K$ choices. Therefore,
Dan will rank the chosen dessert at each value $d \in \{1, \ldots, K \}$ with the same probability
$$P_D(d) \;=\; \frac{1}{K},$$so that
\begin{align}
E(D) \;=\; \frac{K+1}{2}.
\label{edan}
\end{align}

As we can see, Tanya should prefer bigger $K$, while Dan should prefer smaller $K$.

The expected total rank is now given by
$$E(T+D) \;=\; E(T)+E(D) \;=\; \frac{N+1}{K+1} +\frac{K+1}{2}.$$

\begin{theorem} The expected total rank is minimized when $K = \sqrt{2N+2}-1$.  Moreover, this choice of $K$ results in equal expected ranks for each participant so that
$E(T) = E(D)= \sqrt{\frac{N+1}{2}}$.
\label{opt2}
\end{theorem}
In other words, minimizing expected total rank also optimizes {\em fairness,} by minimizing the
absolute difference $|E(T)-E(D)|$ as well.
\begin{proof} From~(\ref{etanya}) and~(\ref{edan}) we have
$$E(T)E(D) \;=\; \frac{N+1}{2},$$
for all choices of $K$.
It then follows from the AM-GM inequality \cite[p. 20]{Steele} that
$$E(T)+E(D) \;\geq \; 2\sqrt{E(T)E(D)} \;=\; \sqrt{2N+2},$$
with equality if and only if $E(T)=E(D)$.  In this optimal equality case, we have
$$\frac{N+1}{K+1} = \frac{K+1}{2},$$
so that $(K+1)^2 = 2N+2$, and $K = \sqrt{2N+2}-1$.
\end{proof}

\section{Rounding Errors}

As $N$ varies from $1$ to $10$ in the formula $K = \sqrt{2N+2}-1$ from Theorem~\ref{opt2}, 
we obtain the following set of values for $K$:
\[1.,\ 1.44949,\ 1.82843,\ 2.16228,\ 2.4641,\ 2.74166,\ 3.,\ 3.24264,\ 3.47214,\ 3.69042.\]

{\bf Remark:} During the real-life coffeehouse encounter that inspired this article, Tanya actually suggested that Dan choose $3$ desserts from a menu of $7$ choices.  We see now that her instincts were on the mark.

In order to maximize the total satisfaction of the parties, Dan should choose $\sqrt{2N+2}-1$ desserts, after which Tanya will indicate her favorite from this group. The problem is that $\sqrt{2N+2}-1$ is usually {\em not an integer}.  

Since we need an integer for this algorithm, we should choose either the floor or the ceiling of $\sqrt{2N+2}-1$. Let $K^\bullet$ denote the integer choice for $K$ that provides the best expected total rank. Here are the first $10$ values of $K^\bullet$, where we use large bold symbols to emphasize the cases when the number is not an approximation, but already an exact integer value:

\renewcommand{\arraystretch}{1.2}
{\centering 
\begin{tabular}{l|c|c|c|c|c|c|c|c|c|c|}
Menu size $N$ \; & \; $1$ & $2$ & $3$ & $4$ & $5$ & $6$ & $7$ & $8$ & $9$ & $10$\\ \hline
Optimal integer $K^\bullet$ \; & \; {\large $\mathbf{1}$} & $1$ or $2$ & $2$ & $2$ & $2$ or $3$ & $3$ & {\large $\mathbf{3}$} & $3$ & $3$ or $4$ & $4$ 
\end{tabular} .
\par}\vspace{2mm}

There is a tie when $N$ is one less than a triangular number: $N +1 = \frac{m(m+1)}{2}= T_m$, where we denote the $m$th triangular number by $T_m$. 
In this case, both $K^\bullet = m-1$ and $K^\bullet = m$ produce the same total rank $m + \frac{1}{2}$.

So, to find the value $K^\bullet$, Dan and Tanya should recall triangular numbers, 
find an index $m$ such that $T_{m-1} < N+1 \leq T_m$, and set $K^\bullet = m-1$.  See Figure~\ref{tri-num}.

\begin{theorem}  
The {\bf integer} number of choices for Dan that optimizes total expected satisfaction for Tanya and Dan; that is,
minimizes $E(T)+E(D)$, is $$K^\bullet = m-1,$$
where $T_{m-1} < N+1 \leq T_m$.

Moreover, using $K^\bullet$ initial choices also assures that
$$|E(T)-E(D)| \leq \frac{1}{2}.$$
\label{approx}
\end{theorem}
In other words, while rounding to $K^\bullet$ does offer an expected advantage to either Dan or Tanya, this advantage does not exceed
$0.5$ in expectation.

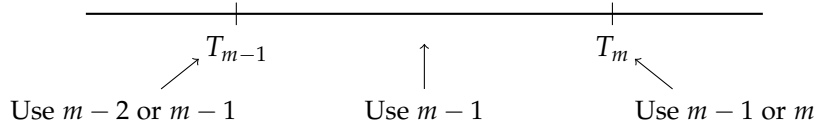
\begin{figure}[ht]
\begin{tikzpicture}[scale=1]
\draw[color=white] (0,-2)--(0,1);
\draw[thick] (0,0)--(9,0);
\draw (2,-0.15)--(2,0.15);
\node [below] at (2,-0.15) {\small $T_{m-1}$};
\draw[->] (1,-1)--(1.5,-0.6);
\node [below] at (0.5,-1) {\footnotesize Use $m-2$ or $m-1$};
\draw (7,-0.15)--(7,0.15);
\node [below] at (7,-0.15) {\small $T_{m}$};
\draw[->] (7.8,-1)--(7.3,-0.6);
\node [below] at (8.5,-1) {\footnotesize Use $m-1$ or $m$};
\draw[->] (4.5,-1)--(4.5,-0.4);
\node [below] at (4.5,-1) {\footnotesize Use $m-1$};
\end{tikzpicture}
\caption{$K^\bullet$ is the largest index of any triangular number less than $N+1$} 
\label{tri-num}
\end{figure}

\begin{proof}[Proof of Theorem~\ref{approx}] 
Suppose that $T_{m-1} < N+1 \leq T_m$.  We then have 
$$m-1 < \sqrt{m(m-1)} < \sqrt{2N+2} \leq \sqrt{m(m+1)} < m+ \frac{1}{2}.$$ 
Therefore, the nearest integers to $K = \sqrt{2N+2}-1$ are $m-2$ and $m-1$.
Using $m-1$ choices yields total expected satisfaction
$$E_{m-1} = E(T)+E(D) = \frac{N+1}{m}+\frac{m}{2}$$
while using  $m-2$ choices yields
$$E_{m-2} = E(T)+E(D) = \frac{N+1}{m-1}+\frac{m-1}{2}.$$
Observe that
$$E_{m-2}-E_{m-1} 
\;=\; (N+1)\left( \frac{1}{m-1}-\frac{1}{m}\right) - \frac{1}{2}
\;=\; (N+1)\left( \frac{1}{2T_{m-1}}\right) - \frac{1}{2} 
\; > \; 0,$$
so that $E_{m-1} <E_{m-2}$. 
Hence, $K^\bullet = m-1$ is the better of the two choices. 

Moreover, after setting $K^\bullet = m-1$, we have
$$E(T)-E(D) \;=\; \frac{N+1}{m}-\frac{m}{2}.$$
Since $T_{m-1} < N+1 \leq T_m$, we have 
$$-\frac{1}{2}  \;=\; \frac{T_{m-1}}{m}-\frac{m}{2} \;<\; E(T)-E(D) \;\leq\;\frac{T_m}{m}-\frac{m}{2} \;=\; \frac{1}{2},$$
so that
$$|E(T)-E(D)| \leq \frac{1}{2}.$$
\end{proof}

{\bf Remark:}  A routine algebraic computation shows that $T_{m-1} \leq N+1 \leq T_m$ if and only if 
$$m = \Big\lfloor \frac{\sqrt{8N+9}+1}{2} \Big\rfloor,$$
leading to the alternative formula
$$K^\bullet = \Big\lfloor \frac{\sqrt{8N+9}-1}{2} \Big\rfloor.$$

For menu sizes $N$ up to $43$ items, the optimal number $K^\bullet$ of initial choices for Dan is given in the following table.\\

\renewcommand{\arraystretch}{1.1}
{\centering 
\begin{tabular}{l|cccccccc}
Menu size $N$ \; & \; $1$ & $2$--$4$ & $5$--$8$ & $9$--$13$ & $14$--$19$ & $20$--$26$ & $27$--$34$ & $35$--$43$\\ \hline
Optimal integer $K^\bullet$ \; & \; $1$& $2$& $3$& $4$& $5$& $6$& $7$& $8$
\end{tabular}
\par}\vspace{1mm}

\section{How effective is the algorithm?}

How far from optimal is the expected total rank using our algorithm? 

In Table~\ref{table:expectation} below we review the expected total rank, $E(T)+E(D)$, resulting from the brute force (Borda count) method (recall the sequence~(\ref{borda})), along with the expectation using the ideal optimal choice value $K = \sqrt{2N+2}-1$ (typically an impractical non-integer), 
and the integer approximation $K^\bullet$ recommended by Theorem~\ref{approx}.

\begin{table}[ht!]
\centering
\begin{tabular}{l|llllllllll}
Menu size $N$: \;& $1$&$2$&$3$&$4$&$5$&$6$&$7$&$8$&$9$&$10$\\ \hline
brute force& $2$& $2.5$& $2.833$& $3.125$& $3.392$& $3.638$& $3.866$& $4.082$& $4.285$& $4.479$\\[1mm]
using ideal $K$& $2$& $2.449$& $2.828$& $3.162$& $3.464$& $3.742$& $4$& $4.243$& $4.472$& $4.690$\\[1mm]
using integer $K^{\bullet}$& $2$& $2.5$& $2.833$& $3.167$& $3.5$& $3.75$& $4$& $4.25$& $4.5$& $4.7$
\end{tabular}\\[4mm]
\caption{Expectation of total rank, $E(T)+E(D)$, for various algorithms}
\label{table:expectation}
\end{table}

We can see from the table above that, for a realistic number of desserts on the menu, the algorithm gives a decent approximation to the optimal choice.

\section{Absolute Difference and Variance}

Theorem~\ref{opt2} asserts that, when $K+1 = \sqrt{2N+2}$, we have $E(T) = E(D)$, suggesting that this choice of $K$ optimizes fairness (and that the integer approximation $K^\bullet$ approximates fairness, as in Theorem~\ref{approx}).  But this is a little bit misleading.

The law of large numbers asserts that, since $E(T) = E(D)$, Dan and Tanya would have equal average satisfaction in the {\em long run}, after many meetings, in which, moreover, Dan's tastes are randomly shuffled each time.  This is not a useful measure of fairness, since we don't go out for dessert that often!  We would like to know instead how $D$ and $T$ are likely to compare on a short-run basis.

To illustrate the difficulty further, suppose $T = 10$ and $D=1$ at one meeting, while $T=1$ and $D=10$ on the next occasion.  Each has the same average satisfaction, in spite of one person being very disappointed each time.

To address this concern, we would like to know how close $T-D$ is likely to be to its expected value $E(T-D) = 0$.  One way to measure this is by computing the variance 
$$\var(T-D) = E\Big[ \Big((T-D)-E(T-D)\Big)^2 \Big]$$ 
and the standard deviation
$\sigma_{T-D} = \sqrt{\var(T-D)}$.  

After setting $K = \sqrt{2N+2}-1$, we have $E(T-D)=0$, so that $$\var(T-D) = E\Big((T-D)^2\Big).$$
In this case the root mean square inequality \cite[p.~36]{Steele} implies 
$$E(|T-D|) \;\leq\; \sqrt{E(|T-D|^2)} \;=\; \sigma_{T-D}.$$
A small upper bound on $\sigma_{T-D}$ would then suggest that $T-D$ is more likely close to zero, rather than merely cancelling out large positive and negative deviations over the long term.

To this end, observe that, since $T$ and $D$ are independent, 
\begin{align*}
E\Big((T-D)^2\Big) \;=\; E\Big( T^2-2TD+D^2 \Big)
\;=\; E( T^2 )+E(D^2)-2E(T)E(D).
\end{align*}

Formulas for $E(T)$ and $E(D)$ were given by~(\ref{etanya}) and~(\ref{edan}), respectively.
Since $D$ is uniform on $\{1,2,\ldots, K\}$,
$$E(D^2) \;=\; \frac{1}{K} \Big( 1^2 + 2^2 + \cdots + K^2 \Big) \;=\; \frac{(K+1)(2K+1)}{6}.$$

To compute $E(T^2)$, observe that 
$$d^2 \;=\; 2 \, \frac{d^2-d}{2} + d \;=\; 2\binom{d}{2} + \binom{d}{1},$$
and recall that $P_T$ is given by~(\ref{tanya}), so that
\begin{align*}
E(T^2) &\;=\; \sum_{d=1}^{N-K+1} P_T(d) d^2
\;=\; \sum_{d=1}^{N-K+1} \binom{N}{K}^{-1} \binom{N-d}{K-1} d^2 \\[1mm]
&\;=\; \binom{N}{K}^{-1} \sum_{d=1}^{N}  \binom{N-d}{K-1}  \left[  2\binom{d}{2} + \binom{d}{1} \right]  \hbox{\quad (the extra terms are all zero) } \\[1mm]
&\;=\; \binom{N}{K}^{-1} \left[  2\binom{N+1}{K+2} + \binom{N+1}{K+1} \right]  \hbox{\quad (after two applications of Proposition~\ref{comb-id}) } \\[1mm]
&\;=\; \frac{K!(N-K)!}{N!} \left[  \frac{2(N+1)!}{(K+2)!(N-K-1)!} +  \frac{(N+1)!}{(K+1)!(N-K)!} \right]  \\[1mm]
&\;=\; \frac{(N+1)(2N-K+2)}{(K+1)(K+2)}. 
\end{align*}

It now follows that
\begin{align*}
E\Big((T-D)^2\Big) 
&\;=\; E( T^2 )+E(D^2)-2E(T)E(D)\\[1mm]
&\;=\; \frac{(N+1)(2N-K+2)}{(K+1)(K+2)}+\frac{(K+1)(2K+1)}{6}-2\frac{N+1}{K+1}\frac{K+1}{2}\\[1mm]
&\;=\; \frac{(N+1)(2N-K+2)}{(K+1)(K+2)}+\frac{(K+1)(2K+1)}{6}-(N+1),
\end{align*}

Setting $K+1 = \sqrt{2N+2}$ now yields
\begin{align*}
E\Big((T-D)^2\Big) 
&\;=\; \frac{(N+1)(2N+3-\sqrt{2N+2})}{\sqrt{2N+2}(\sqrt{2N+2}+1)}+\frac{\sqrt{2N+2}(2\sqrt{2N+2}-1)}{6}-(N+1),\\[1mm]
&\;=\; \frac{(N+1)(2N+3-\sqrt{2N+2})}{(2N+2)+\sqrt{2N+2}}+\frac{2N+2}{3}-\frac{\sqrt{2N+2}}{6}-(N+1)\\[1mm]
&\;<\; \frac{2}{3}(N+1), 
\end{align*}
whence
$$E(|T-D|) \;\leq\; \sigma_{T-D} \; < \; \sqrt{\frac{2}{3}(N+1)}.$$
For example, given a menu of $12$ items, we have  $E(|T-D|) < 3$.

\section{Larger groups of diners}

Suppose instead that a party of $3$ or more graduate students on a very tight budget are sharing one dessert,
which comes from a menu of $N$ different options.  
Let $s$ be the number of people in the group. As before, the desserts will be ranked from top choice (\#$1$) down to the bottom choice (\#N) independently by each person.  Once again, we assume that each person's ranking of different desserts is a uniformly random permutation, independent of the others.

In order to find a dessert for everyone to share,
the first person will be asked to choose $C_1$ dessert options from the menu.  The second person will then choose $C_2$ from that shorter list, and so on, until the last person chooses $C_s=1$ dessert for everyone to share. This sequence of narrowing choices can be visualized as shown:
$$N \longrightarrow C_1 \longrightarrow \cdots \longrightarrow C_{i-i}\longrightarrow C_i\longrightarrow C_{i+1}\longrightarrow \cdots
\longrightarrow C_s = 1.$$

The final dessert chosen will have some (likely different) rank according to each person at the table.
Let $R_i$ denote the rank of the final dessert choice according to the $i$th person.  In order to make a fair choice, we would like the expected values $E(R_i)$ to be as close together as possible.

Recall that the {\em order statistics} of a finite sequence of numbers is an ordered list of those same numbers, sorted from smallest to largest \cite[p.~256]{ross}.
In order to determine each expected rank $E(R_i)$, we will make use of the following well-known fact about order statistics.

\begin{proposition}  Suppose $A$ is a uniformly random $a$-element subset of the integers $\{1,2,\ldots,N\}$.
Let $\{W_1, W_2, \ldots, W_a\}$ denote the order statistics of $A$.  Then
\begin{align*}
E(W_i) \;=\; \frac{i(N+1)}{a+1}.
\end{align*}
\label{expW}
\end{proposition}
\begin{proof}
Consider $N+1$ points evenly spaced around a circle of perimeter $N+1$.
Choose a random $(a+1)$-element subset of these points uniformly.  
It follows from the symmetry of the circle that the expected spacing between any two of these chosen points is the same for each consecutive pair.
Since the sum of these spacings returns the perimeter of the circle, each spacing has the expected value $\frac{N+1}{a+1}$.  Now cut the circle open at the location of the first chosen point after some base point (fixed in advance).  The proposition now follows.
\end{proof}

For the special case we need, suppose $A$ is a uniformly random $a$-element subset of the integers $\{1,2,\ldots,N\}$, and let $B$
denote the largest $b$ values of $A$, so that $B = \{W_{1}, \ldots, W_{b-1}, W_b\}$.  
If a value $\hat{B}$ is chosen uniformly from $B$, then
\begin{align*} 
E(\hat{B}) \;=\; \frac{1}{b}\sum_{i=1}^b E(W_{i}) 
&\;=\; \frac{1}{b}\sum_{i=1}^b \frac{i(N+1)}{a+1} \qquad \hbox{(by~ Proposition~\ref{expW})}\\ 
&\;=\; \frac{N+1}{b(a+1)}\sum_{i=1}^b i \\ 
&\;=\; \frac{N+1}{b(a+1)} \cdot \frac{b(b+1)}{2},
\end{align*}
so that (after simplifying)
\begin{align}
E(\hat{B})\;=\; \frac{(N+1)(b+1)}{2(a+1)}.
\label{expJ}
\end{align}

Returning to our dessert selection,
observe that person $i$ is offered $C_{i-1}$ choices and passes along a number $C_i \leq C_{i-1}$ to the next person.  Person $i$ will
always choose her lowest numbered (highest ranked) $C_i$ choices.  After that, what happens from her perspective is uniformly random.  
So the rank $R_i$ of the final outcome, from the perspective of the $i$th person, is a uniformly random choice of one of the lowest
$C_i$ values, taken from the previous set of size $C_{i-1}$.  It follows from the identity~(\ref{expJ}) that
\begin{align}
E(R_i) \;=\;
\frac{(N+1)(C_i+1)}{2(C_{i-1}+1)}.
\label{expR}
\end{align}
We now wish to set all $E(R_i)$ to be equal, subject to the boundary conditions that $C_0=N$ and $C_s = 1$.  
Suppose $E = E(R_i)$ is this constant rank value.  It then follows that
\begin{align*}
E^s &\;=\; E(R_s)E(R_{s-1})\cdots E(R_1) \\[2mm]
&\;=\;  \frac{(N+1)(C_s+1)}{2(C_{s-1}+1)} \frac{(N+1)(C_{s-1}+1)}{2(C_{s-2}+1)}\cdots
\frac{(N+1)(C_1+1)}{2(C_{0}+1)}\\[2mm]
&\;=\;\left(\frac{N+1}{2}\right)^{s} \frac{C_s+1}{C_0+1}
\;=\;\left(\frac{N+1}{2}\right)^{s} \frac{2}{N+1}
\;=\;\left(\frac{N+1}{2}\right)^{s-1},
\end{align*}
so that the common expected rank will be
\begin{align*}
E = \left(\frac{N+1}{2}\right)^{1-\frac{1}{s}}.
\end{align*}

It follows from the identity~(\ref{expR}) that 
the optimal number $C_i$ of choices from the $i$th participant satisfies
\begin{align*}
C_i+1 &\;=\; 
\frac{2E}{N+1} (C_{i-1}+1) 
\;=\; 
\left(\frac{N+1}{2}\right)^{-\frac{1}{s}} (C_{i-1}+1) \;=\; \cdots \\
&\;=\; \left(\frac{N+1}{2}\right)^{-\frac{i}{s}}(C_{0}+1) 
\;=\; \left(\frac{N+1}{2}\right)^{-\frac{i}{s}}(N+1), 
\end{align*}
so that
\begin{align}
C_i \;=\; 2^\frac{i}{s} (N+1)^\frac{s-i}{s} -1.
\label{optC}
\end{align}
Notice that~(\ref{optC}) satisfies $C_0 = N$, the initial number offered, and $C_s = 1$, as required for the final choice.

We have proven the first part of the following theorem.
\begin{theorem}  A collection of $s$ people with random tastes are choosing to share a dessert from a menu with $N$ initial options.
If each $i$th participant passes $C_i$ options to the next person, where $C_i$ is given by the identity~(\ref{optC}),
then the final outcome will have common expected rank
\begin{align}
E(R_i) \;=\;\left(\frac{N+1}{2}\right)^{1-\frac{1}{s}},
\label{optrank}
\end{align}
that is, the same expected rank for each person.
Moreover, the final choice will have the minimal expected total rank, summed over all participants.
\end{theorem}
In other words, the identity~(\ref{optC}) also minimizes the expected sum
\begin{align}
E(R_1 + \cdots + R_s).
\label{ranksum}
\end{align}
\begin{proof} We have already shown that the rank $R_i$ of the final choice (according to the $i$th participant) has the same expected value given by~(\ref{optrank}) for every participant when the protocol given by~(\ref{optC}) is followed. 

To see why the identity~(\ref{optC}) also minimizes the sum~(\ref{ranksum}), observe that
\begin{align} \notag
\frac{1}{s} E(R_1 &+ \cdots + R_s) 
\;=\; \frac{1}{s} \sum_{i=1}^s \frac{(N+1)(C_i+1)}{2(C_{i-1}+1)} \qquad \hbox{by~(\ref{expR})} \\  
&\;\geq\; \prod_{i=1}^s \left( \frac{(N+1)(C_i+1)}{2(C_{i-1}+1)} \right)^\frac{1}{s}\quad \hbox{by the AM-GM inequality \cite[p. 20]{Steele}} \label{AG} \\ \notag
&\;=\; \frac{N+1}{2} \cdot 
\left( \frac{C_1+1}{C_0+1}\cdot \frac{C_2+1}{C_1+1} \cdots \frac{C_s+1}{C_{s-1}+1} \right)^\frac{1}{s} \\ \notag
&\;=\; \frac{N+1}{2} \cdot \left( \frac{C_s+1}{C_0+1} \right)^\frac{1}{s} 
\;=\; \frac{N+1}{2} \cdot \left( \frac{1+1}{N+1} \right)^\frac{1}{s} 
\;=\; \left(\frac{N+1}{2} \right)^{1-\frac{1}{s}}, 
\end{align}
where equality holds in~(\ref{AG}) if and only if the values $E(R_i)$ are all equal, 
leading once again to~(\ref{optC}) and~(\ref{optrank}).
\end{proof}

Note that, if $s = 2$, we have
$$C_1 \;=\;  2^\frac{1}{2} (N+1)^\frac{1}{2} -1 \;=\;  \sqrt{2N+2} -1, $$
in agreement with our solution to the earlier case of two diners sharing a dessert.

Moreover,
$$\lim_{N \rightarrow \infty} \frac{E(R_i)}{N} \;=\; 0,$$ 
so that the expected percentile rank of this choice approaches $100\%$ for everyone, given a sufficiently large menu.

Of course, we again face the issue that the values $C_i$ prescribed by~(\ref{optC}) will almost never be integers.  The values for $C_i$ used in any practical instance will have to be integer approximations, made at some small cost in terms of optimality and fairness.

For example, suppose that Dan, Tanya, and Eric are choosing one dessert to share from a menu of $12$ options. Dan would choose
$C_1 = 2^{1/3}13^{2/3}-1 \approx 6$ desserts,
from which Tanya would then select
$C_2 = 2^{2/3}13^{1/3}-1 \approx 3$, leaving Eric to select one dessert from among these.

This quick procedure avoids the need for long negotiations or complete menu rankings, leaving more time for  cake and mathematics.

{\scriptsize 
T.~Khovanova, Department of Mathematics, MIT,
77 Massachusetts Avenue,
Cambridge, MA 02139, USA} \\
{\footnotesize {\em Email:} {\tt tanyakh@mit.edu}\par}

{\scriptsize 
D.~Klain, Department of Mathematics and Statistics,
UMass Lowell,
Lowell, MA 01854, USA} \\  {\footnotesize 
{\em Email:} {\tt Daniel\_{}Klain@uml.edu}\par
}
\end{document}